\documentclass[12pt,twoside]{amsart}
\usepackage{amsmath, amsthm, amscd, amsfonts, amssymb, graphicx}
\let
\oldsection
\section
\renewcommand{\section}{\vspace{8pt plus 4pt}\oldsection}
\usepackage{lipsum}
\usepackage[super]{nth}
\usepackage{xcolor}
\definecolor{olive}{rgb}{0.3, 0.4, .1}
\definecolor{fore}{RGB}{249,242,215}
\definecolor{back}{RGB}{51,51,51}
\definecolor{title}{RGB}{255,0,90}
\definecolor{dgreen}{rgb}{0.,0.6,0.}
\definecolor{gold}{rgb}{1.,0.84,0.}
\definecolor{JungleGreen}{cmyk}{0.99,0,0.52,0}
\definecolor{BlueGreen}{cmyk}{0.85,0,0.33,0}
\definecolor{RawSienna}{cmyk}{0,0.72,1,0.45}
\definecolor{Magenta}{cmyk}{0,1,0,0}

\newtheorem{thm}{Theorem}[section]

\newtheorem{corollary}[thm]{Corollary}
\newtheorem{lemma}[thm]{Lemma}

\theoremstyle{definition}
\newtheorem{defn}{Definition}[section]

\theoremstyle{remark}
\newtheorem{rem}{Remark}[section]

\numberwithin{equation}{section}
\begin{document}
\begin{center}\large{{\bf{ Generalized Geometric Difference Sequence Spaces and its duals}}} 
\vspace{0.5cm}

Khirod Boruah$^{1}$ and Bipan Hazarika$^{1,\ast}$ and Mikail Et$^{2}$ 

\vspace{0.5cm}
$^{1}$
Department of Mathematics, Rajiv Gandhi University, Rono Hills, Doimukh-791112, Arunachal Pradesh, India\\
$^{2}$Department of Mathematics, Firat University, 23119, Elazig, Turkey\\

Email: khirodb10@gmail.com; bh\_rgu@yahoo.co.in; mikailet68@gmail.com
\thanks{$^{\ast}$ The corresponding author.}
\end{center}
\title{}
\author{}
\thanks{{24-11-2015}}
\begin{abstract} After introducing the geometric difference sequence spaces in the paper \cite{KhirodBipan}, objective of this paper is to introduce the generalized geometric difference sequence spaces
 $l_\infty^{G}(\Delta^m_G), c^G(\Delta^m_G), c_0^{G}(\Delta^m_G)$
 and to prove that these are Banach spaces. Then we prove some inclusion properties. Also we compute their dual spaces.

\parindent=5mm
\noindent{\footnotesize {\bf{Keywords and phrases:}}} Geometric difference; dual space; geometric integers; geometric complex numbers.\\
{\footnotesize {\bf{AMS subject classification \textrm{(2000)}:}}} 26A06, 11U10, 08A05, 46A45.

\end{abstract}
\maketitle

\maketitle

\pagestyle{myheadings}
\markboth{\scriptsize  K. Boruah, B. Hazarika, M Et}
        {\scriptsize  On Some Generalized Geometric Difference Sequence Spaces}

\maketitle\vspace{-0.4cm}
\section{Introduction}
After the introduction to the "Non-Newtonian Calculus'' by Grossman and Katz \cite{GrossmanKatz}  which is also called as multiplicative calculus, various researchers have been developing its dimensions. The operations of multiplicative calculus is called as multiplicative derivative and multiplicative integral. We refer to Grossman and Katz \cite{GrossmanKatz}, Stanley \cite{Stanley}, Bashirov et al. \cite{BashirovMisirh,BashirovKurpinar}, Grossman  \cite{Grossman83} for elements of multiplicative calculus and
its applications. An extension of multiplicative calculus to functions of complex variables is
handled in Bashirov and R\i za \cite{BashirovRiza}, Uzer \cite{Uzer10}, Bashirov et al. \cite{BashirovKurpinar}, \c{C}akmak and Ba\c{s}ar \cite{CakmakBasar}, Tekin and Ba\c{s}ar\cite{TekinBasar}, T\"{u}rkmen and Ba\c{s}ar \cite{TurkmenBasar}. In \cite{KadakEfe, kadak2} Kadak et al studied the new types of sequences spaces over non-Newtonian Calculus and proved some interesting results. Kadak \cite{Kadak} determined the K\"{o}the-Toeplitz duals over non-Newtonian Complex Field.

Geometric calculus is an alternative to the usual calculus of Newton and Leibniz. It provides differentiation and integration tools based on multiplication instead of addition. Every property in Newtonian calculus has an analog in multiplicative calculus. Generally speaking multiplicative calculus is a methodology that allows one to have a different look at problems which can be investigated via calculus. In some cases, for example for growth related problems, the use of multiplicative calculus is advocated instead of a traditional Newtonian one.

For readers' convenience, it is to be remind that all concepts in classical arithmetic have natural counterparts in $\alpha-arithmetic.$ Consider any generator $\alpha$ with range $A\subseteq \mathbb{C}.$ By $\alpha- arithmetic,$ we mean the arithmetic whose domain is $A$ and operations are defined as follows, 
for $x, y \in A$ and any generator $\alpha,$
\begin{align*}
&\alpha -addition &x\dot{+}y &=\alpha[\alpha^{-1}(x) + \alpha^{-1}(y)]\\
&\alpha-subtraction &x\dot{-}y&=\alpha[\alpha^{-1}(x) - \alpha^{-1}(y)]\\
&\alpha-multiplication &x\dot{\times}y &=\alpha[\alpha^{-1}(x) \times \alpha^{-1}(y)]\\
&\alpha-division &\dot{x/y}&=\alpha[\alpha^{-1}(x) / \alpha^{-1}(y)]\\
&\alpha-order &x\dot{<}y &\Leftrightarrow \alpha^{-1}(x) < \alpha^{-1}(y).
\end{align*}
If we choose ``\textit{$exp$}'' as an $\alpha-generator$ defined by $\alpha (z)= e^z$ for $z\in \mathbb{C}$ then $\alpha^{-1}(z)=\ln z$ and $\alpha-arithmetic$ turns out to geometric arithmetic.
\begin{align*}
&\alpha -addition &x\oplus y &=\alpha[\alpha^{-1}(x) + \alpha^{-1}(y)]& = e^{(\ln x+\ln y)}& =x.y ~geometric ~addition\\
&\alpha-subtraction &x\ominus y&=\alpha[\alpha^{-1}(x) - \alpha^{-1}(y)]&= e^{(\ln x-\ln y)} &=  x\div y, y\ne 0 ~geometric ~subtraction\\
&\alpha-multiplication &x\odot y &=\alpha[\alpha^{-1}(x) \times\alpha^{-1}(y)]& = e^{(\ln x\times\ln y)} & = ~x^{\ln y}~ geometric ~multiplication\\
&\alpha-division &x\oslash y&=\alpha[\alpha^{-1}(x) / \alpha^{-1}(y)] & = e^{(\ln x\div \ln y)}& = x^{\frac{1}{\ln y}}, y\ne 1 ~ geometric ~division.
\end{align*}
In \cite{TurkmenBasar} T\"{u}rkmen and F. Ba\c{s}ar defined the geometric complex numbers $\mathbb{C}(G)$ as follows:
\[\mathbb{C}(G):=\{ e^{z}: z\in \mathbb{C}\} = \mathbb{C}\backslash \{0\}.\]
Then $(\mathbb{C}(G), \oplus, \odot)$ is a field with geometric zero $1$ and geometric identity $e.$\\
Then for all $x, y\in \mathbb{C}(G)$
\begin{itemize}
\item{ $x\oplus y=xy$}
\item{ $x\ominus y=x/y$}
\item{ $x\odot y=x^{\ln y}=y^{\ln x}$}
\item{ $x\oslash y$ or $\frac{x}{y}G=x^{\frac{1}{\ln y}}, y\neq 1$}
\item{$x_1 \oplus x_2 \oplus...\oplus x_n=_G\sum_{i=1}^n x_i=x_1.x_2...x_n$}
\item{ $x^{2_G}= x \odot x=x^{\ln x}$}
\item{ $x^{p_G}=x^{\ln^{p-1}x}$}
\item{ ${\sqrt{x}}^G=e^{(\ln x)^\frac{1}{2}}$}
\item{ $x^{-1_G}=e^{\frac{1}{\log x}}$}
\item{ $x\odot e=x$ and $x\oplus 1= x$}
\item{ $e^n\odot x=x^n=x\oplus x\oplus .....(\text{upto $n$ number of $x$})$}
\item{
\begin{equation*}
\left|x\right|^G=
\begin{cases}
x, &\text{if $x>1$}\\
1,&\text{if $x=1$}\\
\frac{1}{x},&\text{if $x<1$}
\end{cases}
\end{equation*}}
Thus $\left|x\right|^G\geq 1.$
\item{ ${\sqrt{x^{2_G}}}^G=\left|x\right|^G$}
\item{ $\left|e^y\right|^G=e^{\left|y\right|}$}
\item{ $\left|x\odot y\right|^G=\left|x\right|^G \odot \left|y\right|^G$}
\item{ $\left|x\oplus y\right|^G \leq\left|x\right|^G \oplus \left|y\right|^G$}
\item{ $\left|x\oslash y\right|^G=\left|x\right|^G \oslash \left|y\right|^G$}
\item{ $\left|x\ominus y\right|^G\geq\left|x\right|^G \ominus \left|y\right|^G$}
\item{ $0_G \ominus 1_G\odot\left(x \ominus y\right)=y\ominus x\,, i.e.$ in short $\ominus \left(x \ominus y\right)= y\ominus x.$}
\end{itemize}

 Let $l_{\infty},c$ and $c_0$ be the linear spaces of complex bounded, convergent and null  sequences, respectively, normed by
\[||x||_\infty=\sup_k|x_k|.\]

 T\"{u}rkmen and Ba\c{s}ar \cite{TurkmenBasar}  have proved that
\[\omega(G)=\{(x_k): x_k \in \mathbb{C}(G)\, \text{for all}\, k\in \mathbb{N}\}\]
is a vector space over $\mathbb{C}(G)$ with respect to the algebraic operations $\oplus$ addition and $\odot$ multiplication 
\begin{align*}
\oplus : \omega(G) \times \omega (G) &\rightarrow \omega (G)\\
                          (x, y)&\rightarrow x \oplus y =(x_k) \oplus (y_k)=(x_ky_k)\\
\odot : \mathbb{C(G)} \times \omega (G) &\rightarrow \omega (G)\\
													(\alpha, y)&\rightarrow \alpha \odot y=\alpha \odot (y_k)=(\alpha^{\ln y_k}),
\end{align*}
where $x=(x_k), y=(y_k) \in \omega (G)$ and $\alpha \in \mathbb{C}(G).$ Then
\begin{align*}
l_\infty(G) &=\{x=(x_k) \in \omega (G): \sup_{k\in \mathbb{N}}|x_k|^G< \infty\}\\
c(G)        &= \{x=(x_k) \in \omega (G): {_G\lim_{k\rightarrow \infty}}|x_k\ominus l|^G=1\}\\
c_0(G)      &= \{x=(x_k) \in \omega (G): {_G\lim_{k\rightarrow \infty}} x_k=1\}\\
l_p(G)      &= \{x=(x_k) \in \omega (G):{_G\sum^\infty_{k=0}}\left(|x_k|^G\right)^{p_G} <\infty\}, \text{~where ${_G\sum}$ is the geometric sum},
\end{align*}
are classical sequence spaces over the field $\mathbb{C}(G).$ Also they have shown that $l_{\infty}(G),$ $c(G)$ and $c_0(G)$ are Banach spaces with the norm 
\[||x||^{G}=\sup_{k}|x_k|^{G}, x=(x_1,x_2,x_3,...)\in \lambda(G), \lambda\in \{l_{\infty},c, c_0\}.\] Here, $_G\lim$ is the geometric limit defined in \cite{TurkmenBasar}(page no. $27$).
For the convenience, we denoted $l_\infty(G), c(G), c_0(G),$ respectively as $l_\infty^G, c^G, c_0^G.$

In 1981, Kizmaz \cite{Kizmaz} introduced the notion of difference sequence spaces using forward difference operator $\Delta$ and studied the classical difference sequence spaces $\ell _{\infty }(\Delta ),$ $c(\Delta
),$ $c_{0}(\Delta ).$ Following C. T\"{u}rkmen and F. Ba\c{s}ar \cite{TurkmenBasar}, Kizmaz \cite{Kizmaz},  we defined geometric sequence space in \cite{KhirodBipan} as follows:
\[l_\infty^G(\Delta_G)= \{x=(x_k) \in \omega (G): \Delta_G x\in l_\infty^G\},
\text{~where~} {\Delta}_G x=x_k \ominus x_{k+1}.\] where $\Delta_G x= (\Delta_G x_k)= (x_k \ominus x_{k+1}).$ Then we introduced some theorems, definitions and  basic results as follows:
\begin{thm}\label{eight} The space  $l_\infty^{G} \left({\Delta}_G\right)$
is a normed linear space w.r.t. the norm
\begin{equation}\label{eqn44}
\left\|x\right\|^G_{{\Delta}_G}=\left|x_1\right|^G\oplus\left\|{\Delta}_Gx\right\|^G_\infty.
\end{equation}
\end{thm}
\begin{thm} The space $l_\infty^{G} \left({\Delta}_G\right)$
is a Banach space w.r.t. the norm $\left\|.\right\|^G_{{\Delta}_G}.$
\end{thm}
\begin{rem}
The spaces \begin{enumerate}
\item[(a)] $c^{G}(\Delta_{G})=\{(x_k)\in w(G): \Delta_{G}x_k\in c^{G}\}$
\item[(b)] $c_{0}^{G}(\Delta_{G})=\{(x_k)\in w(G): \Delta_{G}x_k\in c_{0}^{G}\}$
\end{enumerate} 
 are Banach spaces with respect to the norm $||.||^{G}_{\Delta_G}.$ Also these spaces are BK-spaces.
\end{rem}
\begin{lemma}\label{1}
The following conditions (a) and (b) are equivalent:
\begin{align*}
&(a) \sup_k|x_k\ominus x_{k+1}|^G<\infty ~~ i.e. ~~ \sup_k|\Delta_G x_k|^G<\infty;\\
&(b)(i) \sup_k e^{k^{-1}}\odot|x_k|^G<\infty \text{~and}\\
& \quad (ii)\sup_k|x_k\ominus e^{{k(k+1)}^{-1}}\odot x_{k+1}|^G<\infty. 
\end{align*}
\end{lemma}
\begin{lemma}\label{2}
\[\text{If} ~ \sup_n\left|{_G\sum^n_{v=1}}c_v\right|^G\leq \infty \text {~then~} \sup_n\left(p_n\odot\left|{_G\sum^\infty_{k=1}}\frac{c_{n+k-1}}{p_{n+k}}G\right|^G\right)<\infty.\]
\end{lemma}
\begin{corollary}\label{Cor1} Let $(p_n)$ be monotonically increasing. If
$$\sup_n\left|{_G\sum_{v=1}^n} p_v\odot a_v\right|^G<\infty \text{~then~} \sup_n\left|p_n\odot {_G\sum_{k=n+1}^\infty} a_k\right|^G<\infty.$$
\end{corollary}
\begin{corollary}\label{Cor2}
$$\text{If}~ {_G\sum_{k=1}^\infty} p_k\odot a_k \text{~is convergent~ then~} \lim_n p_n\odot {_G\sum_{k=n+1}^\infty} a_k=1.$$
\end{corollary}
\begin{corollary}\label{Cor3}
 $${_G\sum_{k=1}^\infty} e^k\odot a_k \text{~is convergent iff~}{_G\sum_{k=1}^\infty} R_k\text{~is convergent with~} e^n\odot R_n = O(e),\text{~where~}$$
 $$R_n = {_G\sum_{k=n+1}^\infty} a_k.$$
\end{corollary}
\begin{defn} \cite{Garling67, KotheToplitz69, KotheToplitz34, Maddox80} If $X$ is a sequence space, we define
\begin{enumerate}
\item[(i)] $X^\alpha=\{a=(a_k) : \sum_{k=1}^\infty |a_k x_k|<\infty, \, \text{for each} \,x\in X\};$
\item[(ii)] $X^\beta=\{a=(a_k) : \sum_{k=1}^\infty a_k x_k \, \text{is convergent, for each} \,x\in X\};$
\item[(iii)] $ X^\gamma=\{a=(a_k) : \sup_n|\sum_{k=1}^n a_k x_k|<\infty, \, \text{for each} \,x\in X\}.$
\end{enumerate}
\end{defn}
\begin{thm}
\[(i)~\text{If~} D_1= \left\{a=(a_k): {_G\sum_{k=1}^\infty} e^k\odot |a_k|^G<\infty\right\} \text{~then~} \left(sl_\infty^G(\Delta_G)\right)^\alpha=D_1.\]
\[(ii)~\text{If~} D_2= \left\{a=(a_k): {_G\sum_{k=1}^\infty} e^k\odot a_k \text{~is convergent with~} {_G\sum_{k=1}^\infty}|R_k|^G<\infty\right\}.\]
Then $\left(sl_\infty^G(\Delta_G)\right)^\beta=D_2.$
\[(iii)~\text{If~} D_3= \left\{a=(a_k): \sup_n|{_G\sum_{k=1}^n} e^k\odot a_k|^G<\infty, {_G\sum_{k=1}^\infty}|R_k|^G<\infty\right\}.\]
Then $\left(sl_\infty^G(\Delta_G)\right)^\gamma=D_3.$
\end{thm}

 Following Kizmaz, generalized sequence spaces $l_\infty(\Delta^m), c(\Delta^m)$ and $c_0(\Delta^m)$ were introduced by Mikail Et and Rifat \c{C}olak \cite{MikailColak}.

\section{Main Results}
Following Mikail Et and Rifat \c{C}olak \cite{MikailColak,MikailColak97} and our own paper\cite{KhirodBipan}, now we define the following new sequence spaces 
\begin{align*}
l^G_\infty(\Delta^m_G)&=\{x=(x_k):\Delta^m_G x\in lG_\infty\},\\
c^{G}(\Delta^m_G)&=\{x=(x_k):\Delta^m_G x\in c^{G}\},\\
l^G_0(\Delta^m_G)&=\{x=(x_k):\Delta^m_G x\in c^G_0\}.
\end{align*}
where $m \in \mathbb{N}~$ and
\begin{align*}
\Delta^0_G x &= (x_k)\\
\Delta_G x   &=(\Delta_G x_k) =(x_k \ominus x_{k+1})\\
\Delta^2_G x &= (\Delta^2_G x_k)= (\Delta_G x_k \ominus \Delta_G x_{k+1})\\
             &= (x_k \ominus x_{k+1} \ominus x_{k+1} \oplus x_{k+2})\\
						 &=(x_k \ominus e^2\odot x_{k+1} \oplus x_{k+1})\\				
\Delta^3_G x &=(\Delta^3_G x_k)= (\Delta^2_G x_k \ominus \Delta^2_G x_{k+1})\\ 
             &=(x_k \ominus e^3\odot x_{k+1} \oplus e^3\odot x_{k+1} \ominus x_{k+3})\\
.......&.........................................................\\
\Delta^m_G x &=(\Delta^m_G x_k)= (\Delta^{m-1}_G x_k \ominus \Delta^{m-1}_G x_{k+1})\\
             &=\left( _G\sum^m_{v=0} (\ominus e )^{{v}_G}\odot e^{\binom{m}{v}}\odot x_{k+v}\right),~\text{with}~(\ominus e)^{0_G}=e.                                 
\end{align*}
Then it can be easily proved that $l_\infty^{G}(\Delta^m_G), c_\infty^{G}(\Delta^m_G)$  and $c_0^{G}(\Delta^m_G)$ are normed linear spaces with norm
\[ \left\|x\right\|^G_{{\Delta}_G}= _G\sum^m_{i=1}\left|x_i\right|^G\oplus\left\|{\Delta}^m_G x\right\|^G_\infty. \]
\textbf{Note}: Throughout this paper often we write $_G\sum_k$ instead of $_G\sum_{k=1}^{\infty}$ and $\lim_k$ instead of $\lim_{n \rightarrow \infty}.$

\begin{defn}[\textbf{Geometric Associative Algebra}]

An associative algebra is a vector space $A\subset \mathbb{R}(G),$
equipped with a bilinear map(called multiplication)
\begin{align*}
\odot :A \times A &\rightarrow A\\
(a, b) &\rightarrow a\odot b
\end{align*}
which is associative, i.e.
\begin{align*}
(a\odot b)\odot c &= a\odot(b\odot c) \forall a, b, c \in A.
\end{align*}
An algebra is commutative if $a\odot b = b\odot a$ for all $a, b \in A.$ An
algebra is unital if there exists a unique $e\in A$ such that $e\odot a = a\odot e = a$ for all $a \in A.$ A subalgebra of the algebra $A$ is a subspace $B$ that is closed
under multiplication, i.e. $a\odot b \in A$ for all $a, b \in B.$
\end{defn}
\begin{defn}[\textbf{Geometric Normed Algebra}]
A normed algebra is a normed space $A\subset \omega(G)$ that is also an associative algebra,
such that the norm is submultiplicative: $\left\|a\odot b\right\|^G \le \left\|a\right\|^G\odot \left\|b\right\|^G$ for all $a, b \in A.$ A geometric algebra is a complete normed algebra, i.e., a normed algebra which is also a Banach space
with respect to its norm.
\end{defn}
It is to be noted that the submultiplicativity of the norm means that multiplication in normed algebras is jointly continuous, i.e. if $a_n \xrightarrow{G} a$ and $b_n \xrightarrow{G} b$ then $(a_n)$ is bounded and
\begin{align*}
\left\|a_n\odot b_n\ominus a\odot b\right\|^G&= \left\|a_n\odot(b_n\ominus b)\oplus (a_n \ominus a)\odot b\right\|^G \\
&\le  \left\|a_n\right\|^G\odot\left\|b_n\ominus b\right\|^G\oplus \left\|a_n \ominus a\right\|^G\odot \left\|b\right\|^G\\
&\le\sup\left\{\left\|a_n\right\|^G\right\}\odot\left\|b_n\ominus b\right\|^G\oplus \left\|b\right\|^G\odot \left\|a_n \ominus a\right\|^G \xrightarrow{G}1 \text{~as~$n\rightarrow \infty$}.
\end{align*}
\begin{defn}[\textbf{Geometric Sequence  Algebra}] A geometric sequence space $E(G)$ is said to be sequence algebra if $x\odot y \in E(G)$ for $x=(x_k), y=(y_k) \in E(G).$ i.e. $E(G)$ is closed under the geometric multiplication $\odot$ defined by
 \begin{align*}
\odot : E(G) \times E(G) &\rightarrow E(G)\\
													(x, y)&\rightarrow x \odot y=(x_k) \odot (y_k)=(x_k^{\ln y_k})
\end{align*}
for any two sequences $x=(x_k), y=(y_k) \in E.$
\end{defn} 
Since $\omega(G)$ is closed under geometric multiplication $\odot,$ hence, $\omega(G)$ is a sequence algebra. Also equence algebra $\omega(G)$ is unital as
$\left\|e_G\right\|^G = e,$ where $e_G=(e, e, e,.......)\in \omega(G).$ 

\begin{defn}[\textbf{Continuous Dual Space}] If $X$ is a normed space, a linear map $f:X\rightarrow \mathbb{R}(G)$ is called linear functional. $f$ is called continuous linear functional or bounded linear functional if $\left\|f\right\|^G<\infty ,$ where
\begin{equation*}
\left\|f\right\|^G= \sup\left\{\left|f(x)\right|^G: \left\|x\right\|^G\le e, \text{~for all~} x\in X\right\}
\end{equation*}
 Let $X^*$ be the collection of all bounded linear functionals on $X.$ If $f,g \in X^*$ and $\alpha \in \mathbb{R}(G),$ we define $(\alpha \odot f \oplus g)(x)=\alpha \odot f(x) \oplus g(x); X^*$ is called the continuous dual space of $X.$
\end{defn}
\begin{thm}
The sequence spaces $l_\infty^G(\Delta^m_G), c^G(\Delta^m_G)$ and $c_0^G(\Delta^m_G)$ are Banach spaces with the norm
\[ \left\|x\right\|^G_{{\Delta}_G}= _G\sum^m_{i=1}\left|x_i\right|^G\oplus\left\|{\Delta}^m_G x\right\|^G_\infty. \]
\end{thm}
\begin{proof}
Let $(x_n)$ be a Cauchy sequence in $l_\infty^G(\Delta^m_G),$ where $x_n=(x_i^{(n)})= (x_1^{(n)}, x_2^{(n)}, x_3^{(n)}, ....)$ for $n \in \mathbb{N}$ and $x_k^{(n)}$ is the $k^{\text{th}}$ coordinate of $x_n.$ Then
\begin{gather}\label{101}
\begin{aligned}
\left\|x_n\ominus x_l\right\|^G_{{\Delta}_G}&= _G\sum_{i=1}^m\left|x_i^{(n)}\ominus x_i^{(l)}\right|^G \oplus \left\|{\Delta}^m_G (x_n\ominus x_l)\right\|^G_\infty\\
&= _G\sum_{i=1}^m\left|x_i^{(n)}\ominus x_i^{(l)}\right|^G \oplus \sup_k|{\Delta}^m_G (x_n\ominus x_l)|^G \rightarrow 1 \text{~as $l, n\rightarrow \infty.$}
\end{aligned}
\end{gather}
Hence we obtain
\[ |x_k^{(n)} \ominus x_k^{(l)}|^G \rightarrow 1\]
as $n, l\rightarrow \infty$ and for each $k \in \mathbb{N}.$ Therefore $(x_k^{(n)})= (x_k^{(1)}, x_k^{(2)}, x_k^{(3)},.....)$ is a Cauchy sequence in $\mathbb{C}(G).$ Since $\mathbb{C}(G)$ is complete, $(x_k^{(n)})$ is convergent.

Suppose $\lim_n x_k^{(n)}= x_k,$ for each $k \in \mathbb{N}.$ Since $(x_n)$ is a Cauchy sequence, for each $\epsilon > 1,$ there exists $N=N(\epsilon)$ such that $\left\|x_n \ominus x_l\right\|^G_{\Delta_G} < \epsilon$ for all $n, l \geq N.$ Hence from (\ref{101})
\[ _G\sum_{i=1}^m \left|x_i^{(n)} \ominus x_i^{(l)}\right|^G < \epsilon~ \text{and}~ \left|_G\sum^m_{v=0} (\ominus e )^{{v}_G}\odot e^{\binom{m}{v}}\odot (x_{k+v}^{(n)} \ominus x_{k+v}^{(l)})\right|^G < \epsilon \]
for all $k\in \mathbb{N}$ and  $n,l\geq N.$ So we have
\begin{align*}
\qquad &\lim_l~  _G\sum^m_{i=1} \left|x_i^{(n)} \ominus x_i^{(l)}\right|^G = _G\sum^m_{i=1} \left|x_i^{(n)} \ominus x_i\right|^G < \epsilon\\
\text{and}& \lim_l \left| _G\Delta^m_G (x_k^{(n)} \ominus x_k^{(l)})\right|^G = \left| _G\Delta^m_G (x_k^{(n)} \ominus x_k)\right|^G  < \epsilon ~ \forall n \geq N.
\end{align*}
This implies $\left\|x_n \ominus x\right\|^G_{\Delta_G} < \epsilon^2 ~\forall n \geq N,$ that is $x_n \stackrel{G}{\rightarrow} x$ as $n\rightarrow \infty,$ where $x=(x_k).$ Now we have to show that $x \in l^G_\infty(\Delta^m_G).$  We have
\begin{align*}
\left|\Delta^m_G x_k\right|^G &= \left|~_G\sum^m_{v=0} (\ominus e)^{v_G} \odot e^{\binom{m}{v}} \odot x_{k+v}\right|^G\\
                              &= \left|~_G\sum^m_{v=0} (\ominus e)^{v_G} \odot e^{\binom{m}{v}} \odot (x_{k+v} \ominus x^N_{k+v} \oplus x^N_{k+v})\right|^G\\
															& \leq \left|~_G\sum^m_{v=0} (\ominus e)^{v_G} \odot e^{\binom{m}{v}} \odot (x^N_{k+v} \ominus x_{k+v})\right|^G \oplus \left|~_G\sum^m_{v=0} (\ominus e)^{v_G} \odot e^{\binom{m}{v}} \odot x^N_{k+v}\right|^G\\
															&\leq \left\| x^N \ominus x\right\|^G_{\Delta_G} \oplus \left|\Delta^m_G ~x^N_k\right|^G = \text{O}(e).
\end{align*}
Therefore we obtain $x \in l^G_\infty (\Delta^m_G).$ Hence $l^G_\infty (\Delta^m_G)$ is a Banach space.
\end{proof}
It can be shown that $c^G(\Delta^m_G)$ and $c^G_0(\Delta^m_G)$ are closed subspaces of $l^G_\infty (\Delta^m_G).$ Therefore these sequence spaces are Banach spaces with the same norm defined for $l^G_\infty (\Delta^m_G),$ above.

Now we give some inclusion relations between these sequence spaces.
\begin{lemma}\label{lemma1st}
\begin{align*}
&(i) \quad c^G_0(\Delta^m_G) \subsetneq c^G_0(\Delta^{m+1}_G);\\
&(ii) \quad c^G(\Delta^m_G) \subsetneq c^G(\Delta^{m+1}_G);\\
&(iii) \quad l^G_\infty (\Delta^m_G) \subsetneq l^G_\infty (\Delta^{m+1}_G).
\end{align*}
\end{lemma}
\begin{proof}
$(i)$ Let $x \in c^G_0(\Delta^m_G).$ Since 
\begin{align*}
\left|\Delta^{m+1}_G x_k\right|^G &= \left|\Delta^m_G x_k \ominus \Delta^m_G x_{k+1}\right|^G\\
                                  & \leq \left|\Delta^m_G x_k \right|^G \oplus \left|\Delta^m_G x_{k+1}\right|^G \rightarrow 1 ~\text{as}~ k\rightarrow \infty.
\end{align*}
$\therefore$ we obtain $x \in c^G_0(\Delta^{m+1}_G).$ Thus $c^G_0(\Delta^m_G) \subset c^G_0(\Delta^{m+1}_G).$

 This inclusion is strict. For let 
\[ x= (e^{k^m})= (e, e^{2^m},e^{3^m},e^{4^m},....., e^{k^m},.......).\]
Then $ x \in c^G_0(\Delta^{m+1}_G)$ as $(m+1)^{\text{th}}$ geometric difference of $e^{k^m}$ is 1(geometric zero). But $ x \notin c^G_0(\Delta^m_G)$ as $m^{\text{th}}$ geometric difference of $e^{k^m}$ is a constant. Hence the inclusion is strict.

The proofs of $(ii)$ and $(iii)$ are similar to that of $(i).$
\end{proof}
\begin{lemma}
\begin{align*}
&(i) \quad c^G_0(\Delta^m_G) \subsetneq c^G(\Delta^m_G);\\
&(ii) \quad c^G(\Delta^m_G) \subsetneq l^G_\infty (\Delta^m_G).
\end{align*}
\end{lemma}
Proofs are similar to that of Lemma (\ref{lemma1st}).

Furthermore, since the sequence spaces $l_\infty^G(\Delta^m_G), c^G(\Delta^m_G)$ and $c_0^G(\Delta^m_G)$ are Banach spaces with continuous coordinates, that is, $\left\|x_n \ominus x\right\|^G_{\Delta_G}\rightarrow 1$ implies $\left|x^{(n)}_k \ominus x_k\right|^G \rightarrow 1 \forall k\in \mathbb{N}$ as $n \rightarrow \infty,$ these are also BK-spaces.
\begin{rem}
It can be easily proved that $c^G_0$ is a sequence algebra. But in general, $l_\infty^G(\Delta^m_G), c^G(\Delta^m_G)$ and $c_0^G(\Delta^m_G)$ are not sequence algebra. For let $x= (e^k), y=(e^{k^{m-1}}).$ Clearly $x, y \in c_0^G(\Delta^m_G).$ But
\[x\odot y= \left(e^k \odot e^{k^{m-1}}\right)= \left(e^{k^m}\right) \notin c_0^G(\Delta^m_G) ~\text{for}~ m\geq 2,\]
since $m^{\text{th}}$ geometric difference of $e^{k^m}$ is constant. 
\end{rem}
Let us define the operator
\[D: l_\infty^G(\Delta^m_G) \rightarrow l_\infty^G(\Delta^m_G) \quad \text{as}\]
$Dx= (1, 1, 1,...,1,  x_{m+1}, x_{m+2},...),$ where $x= (x_1, x_2, x_3,...,x_m, x_{m+1},...) \in l_\infty^G(\Delta^m_G).$ It is trivial that $D$ is a bounded linear operator on $l_\infty^G(\Delta^m_G).$ Furthermore, the set
\[D\left[l_\infty^G(\Delta^m_G)\right]= Dl_\infty^G(\Delta^m_G)=\{x=(x_k): x \in l_\infty^G(\Delta^m_G), x_1=x_2=...=x_m=1\}\]
is a subspace of $l_\infty^G(\Delta^m_G)$ and 
\begin{align*}
\left\|x\right\|^G_{\Delta_G}&=\left|x_1\right|^G\oplus \left|x_2\right|^G\oplus ... \oplus \left|x_m\right|^G\oplus \left\|\Delta^m_Gx\right\|^G_\infty \\
                             &= 1\oplus 1\oplus ...\oplus 1\oplus \left\|\Delta^m_Gx\right\|^G_\infty \\
														&= \left\|\Delta^m_Gx\right\|^G_\infty\\
\therefore \left\|x\right\|^G_{\Delta_G}&= \left\|\Delta^m_Gx\right\|^G_\infty ~\text{in}~ Dl_\infty^G(\Delta^m_G).
\end{align*}
Now let us define
\begin{gather}\label{eqn1}
\begin{aligned}
\Delta^m : Dl_\infty^G(\Delta^m_G) &\rightarrow l_\infty^G \\
                            \Delta^m_G x&=y =(\Delta^{m-1}_G x_k \ominus \Delta^{m-1}_G x_{k+1} ).      
\end{aligned}
\end{gather}
\textbf{$\Delta^m_{G}$ is a linear homomorphism:} Let $x, y \in Dl_\infty^G(\Delta^m_G).$ Then
\begin{align*}
\Delta^m_G (x_k \oplus y_k) &=  _G\sum^m_{v=0} (\ominus e )^{{v}_G}\odot e^{\binom{m}{v}}\odot(x_k \oplus y_k)\\
                            &= _G\sum^m_{v=0} (\ominus e )^{{v}_G}\odot e^{\binom{m}{v}}\odot x_k \oplus _G\sum^m_{v=0} (\ominus e )^{{v}_G}\odot e^{\binom{m}{v}}\odot y_k\\
														&=\Delta^m_G x_k\oplus \Delta^m_G y_k\\
					\therefore \quad \Delta^m_G (x \oplus y) &= \Delta^m_G x \oplus \Delta^m_G y.~ \text{For}~ \alpha \in \mathbb{C}(G)\\
\Delta^m_G (\alpha \odot x) &= \left(\Delta^m_G \alpha \odot x_k\right)\\
                                             &= \left(_G\sum^m_{v=0} (\ominus e )^{{v}_G}\odot e^{\binom{m}{v}}\odot \alpha \odot x_k\right)\\
																						 &= \left( \alpha \odot _G\sum^m_{v=0} (\ominus e )^{{v}_G}\odot e^{\binom{m}{v}}\odot x_k\right)\\
&= \alpha \odot \Delta^m_G\odot x.
\end{align*}
This implies that $\Delta^m_G $ is a linear homomorphism. Hence $Dl^G_\infty(\Delta^m_G)$ and $l^G_\infty $ are equivalent as topological spaces \cite{Maddox80}. $\Delta^m_G$ and $(\Delta^m_G)^{-1}$ are norm preserving and
\[\left\|\Delta^m_G\right\|^G_\infty=\left\|(\Delta^m_G)^{-1}\right\|^G_\infty=e. \]
Let $\left[l^G_\infty\right]^{'}$ and $\left[ Dl^G_\infty(\Delta^m_G)\right]^{'}$ denote the continuous duals of $l^G_\infty $ and $Dl^G_\infty(\Delta^m_G),$ respectively.

It can be shown that
\begin{align*}
s: \left[ Dl^G_\infty(\Delta^m_G)\right]^{'} &\rightarrow \left[l^G_\infty\right]^{'}\\
                                          f_{\Delta} &\rightarrow f_{\Delta} \circ (\Delta^m_G)^{-1}=f
\end{align*}
is a linear isometry. So $\left[ Dl^G_\infty(\Delta^m_G)\right]^{'}$ is equivalent to $\left[l^G_\infty\right]^{'}.$

In the same way, it can be shown that $Dc^G(\Delta^m_G)$ and  $Dc^G_0(\Delta^m_G)$ are equivalent as topological space to $c^G$ and $c^G_0 $, respectively. Also
\[ \left[Dc^G(\Delta^m_G)\right]^{'}\cong \left[Dc^G_0(\Delta^m_G)\right]^{'} \cong l^G_1, \]
where $l^G_1= \{x=(x_k): _G\sum_k |x_k|^G < \infty\}$.\vspace{-0.4cm}
\section{Dual spaces of $l^G_\infty(\Delta^m_G)$ and $c^G(\Delta^m_G)$}
In this section we construct the $\alpha$-dual spaces of $l^G_\infty(\Delta^m_G)$ and $c^G(\Delta^m_G).$ Also we show that these spaces are not perfect spaces.
\begin{lemma}\label{lemma3rd}
The following conditions (a) and (b) are equivalent:
\begin{align*}
&(a) \sup_k|x_k\ominus x_{k+1}|^G<\infty ~~ i.e. ~~ \sup_k|\Delta_G x_k|^G<\infty;\\
&(b)(i) \sup_k e^{k^{-1}}\odot|x_k|^G<\infty \text{~and}\\
& \quad (ii)\sup_k|x_k\ominus e^{{k(k+1)}^{-1}}\odot x_{k+1}|^G<\infty. 
\end{align*}
\end{lemma}
\begin{proof}
Let (a) be true i.e. $\sup_k|x_k\ominus x_{k+1}|^G<\infty .$
\begin{align*}
\text{Now~} |x_1\ominus x_{k+1}|^G&=\left|{_G\sum^k_{v=1}}{(x_v \ominus x_{v+1})}\right|^G\\
 &=\left|{_G\sum^k_{v=1}}{\Delta_Gx_v}\right|^G\\
&\leq {_G\sum^k_{v=1}}\left|\Delta_Gx_v\right|^G=O(e^k)\\
	\text{and~} |x_k|^G  &=|x_1\ominus x_1\oplus x_{k+1}\oplus x_k\ominus x_{k+1}|^G\\
	 &\leq |x_1|^G\oplus |x_1\ominus x_{k+1}|^G\oplus |x_k\ominus x_{k+1}|^G=O(e^k).			
	 \end{align*}
This implies that $\sup_k e^{k^{-1}}\odot|x_k|^G<\infty.$ This completes the proof of $b(i).$\\
Again
\begin{align*}
\sup_k\left|x_k\ominus e^{{k(k+1)}^{-1}}\odot x_{k+1}\right|^G &=\left|\left\{e^{{(k+1)}}\odot e^{{(k+1)}^{-1}}\right\}\odot x_k\ominus e^{{k(k+1)}^{-1}}\odot x_{k+1} \right|^G \\
  &=\left|\left\{(e^k \oplus e)\odot e^{{(k+1)}^{-1}}\right\}\odot x_k\ominus e^{{k(k+1)}^{-1}}\odot x_{k+1} \right|^G \\
  &=\left|\left\{e^{k(k+1)^{-1}}\odot x_k\oplus e^{(k +1)^{-1}}\odot x_k \right\}\ominus e^{k(k+1)^{-1}}\odot x_{k+1}\right|^G\\
	&=\left|\left\{e^{k(k+1)^{-1}}\odot(x_k\ominus x_{k+1})\right\}\oplus \left\{e^{(k+1)^{-1}}\odot x_k\right\}\right|^G\\
	&\leq e^{k(k+1)^{-1}}\odot \left|x_k\ominus x_{k+1}\right|^G\oplus e^{(k+1)^{-1}}\odot \left|x_k\right|^G\\
			&=O(e).
\end{align*}
Therefore $\sup_k|x_k\ominus e^{{k(k+1)}^{-1}}\odot x_{k+1}|^G<\infty.$ This completes the proof of $b(ii).$

Conversely let $(b)$ be true. Then
\begin{align*}
\left|x_k\ominus e^{k(k+1)^{-1}}\odot x_{k+1}\right|^G&=\left|e^{(k+1)(k+1)^{-1}}\odot x_k\ominus e^{k(k+1)^{-1}}\odot x_{k+1}\right|^G\\
                                                         &\geq e^{k(k+1)^{-1}}\odot|x_k\ominus x_{k+1}|^G\ominus e^{(k+1)^{-1}}\odot |x_k|^G
\end{align*}
i.e. $e^{k(k+1)^{-1}}\odot|x_k\ominus x_{k+1}|^G\leq e^{(k+1)^{-1}}\odot |x_k|^G\oplus \left|x_k\ominus e^{k(k+1)^{-1}}\odot x_{k+1}\right|^G.$\\
Thus $\sup_k|x_k\ominus x_{k+1}|^G<\infty$ as $b(i)$ and $b(ii)$ hold.
\end{proof}
\begin{corollary}\label{corollary1} The following conditions $(a)$ and $(b)$ are equivalent
\begin{align*}
&(a) \sup_k \left|\Delta^{m-1}_G x_k \ominus \Delta^{m-1}_G x_{k+1}\right|^G < \infty;\\
&(b) (i) \sup_k e^{k^{-1}}\odot\left|\Delta^{m-1}_G x_k\right|^G < \infty\\
& \quad (ii) \sup_k \left|\Delta^{m-1}_G x_k \ominus e^{{k(k+1)}^{-1}}\odot \Delta^{m-1}_G x_{k+1}\right|^G < \infty.
\end{align*} 
\end{corollary}
\begin{proof}
By putting $\Delta^{m-1}_G x_k$ instead of $x_k$ in Lemma (\ref{lemma3rd}), results are obvious.
\end{proof}
\begin{lemma}\label{lemma4th}
\[\sup_k e^{k^{-i}}\odot \left|\Delta_G x_k\right|^G < \infty ~\text{implies}~ \sup_k e^{^{-(i+1)}}\odot |x_k|^G<\infty ~\forall i\in \mathbb{N}.\]
\end{lemma}
\begin{proof}
For $i=1$ it is obvious from the Lemma (\ref{lemma3rd}). Let the result be true for $i=n.$ i.e. $\sup_k e^{k^{-n}}\odot|\Delta_G x_k|^G<\infty.$ Then
\begin{align*}
|x_k\ominus x_{k+1}|^G&=|  _G\sum^k_{v=1}\Delta_G x_k|^G\\
                                          &\leq~ _G\sum^k_{v=1}|\Delta_G x_k|^G=\text{O}\left({ \left(e^{k^n}\right)}^k\right)=\text{O}\left(e^{k^{(n+1)}}\right), ~\text{as}~ \sup_k e^{k^{-n}}\odot|\Delta_G x_k|^G<\infty\\
\text{and}~~ |x_k|^G &=|x_k \oplus x_1\ominus x_1\oplus x_{k+1}\ominus x_{k+1}|^G\\
                                      &\leq |x_1|^G \oplus |x_1 \ominus x_{k+1}|^G\oplus |x_k \ominus x_{k+1}|^G= \text{O}\left(e^{k^{(n+1)}}\right).
\end{align*}
From this we obtain, $\sup_k e^{k^{-(n+1)}} \odot |x_k|^G< \infty.$
Thus $\sup_k e^{k^{-(i+1)}} \odot |x_k|^G< \infty \quad\forall i\in \mathbb{N}.$
\end{proof}
\begin{lemma}\label{lemma5th}
\begin{align*}
&\sup_k e^{k^{-i}}\odot \left|\Delta^{m-1}_G x_k\right|^G < \infty ~\text{implies}\\
&\sup_k e^{^{-(i+1)}}\odot |\Delta^{m-{(i+1)}}_G x_k|^G<\infty ~\forall i, m \in \mathbb{N}~\text{and}~ 1\leq i <m.
\end{align*}
\end{lemma}
\begin{proof}
Putting $\Delta^{m-i}_G x_k$ instead of $\Delta_G x_k$ in Lemma (\ref{lemma4th}), the result is immediate.
\end{proof}
\begin{corollary}\label{corollary2}
$\sup_k e^{k^{-1}}\odot\left|\Delta^{m-1}_G x_k\right|<\infty$ implies $\sup_k e^{k^{-m}}\odot\left|x_k\right|<\infty.$
\end{corollary}
\begin{proof}
In Lemma (\ref{lemma5th}) putting $i=1$, we get
\begin{align*}
&\sup_k e^{k^{-1}}\odot \left|\Delta^{m-1}_G x_k\right|^G <\infty &\Rightarrow \sup_k e^{k^{-2}}\odot \left|\Delta^{m-2}_G x_k\right|^G <\infty\\
\text{Similarly,}&&\\
&\sup_k e^{k^{-2}}\odot \left|\Delta^{m-2}_G x_k\right|^G <\infty &\Rightarrow \sup_k e^{k^{-3}}\odot \left|\Delta^{m-3}_G x_k\right|^G <\infty.
\end{align*}
Continuing the process we get
\begin{align*}
&\sup_k e^{k^{-(m-1)}}\odot \left|\Delta^1_G x_k\right|^G <\infty &\Rightarrow \sup_k e^{k^{-m}}\odot \left|\Delta^0_G x_k\right|^G <\infty\\
\text{Thus}\quad &\sup_k e^{k^{-1}}\odot\left|\Delta^{m-1}_G x_k\right|<\infty &\Rightarrow\sup_k e^{k^{-m}}\odot\left|x_k\right|<\infty.
\end{align*}
\end{proof}
\begin{corollary}\label{corollary3}
If $x\in l^G_\infty(\Delta^m_G)$~ then ~ $\sup_k e^{k^{-m}}\odot|x_k|^G<\infty.$
\end{corollary}
\begin{proof}
\begin{align*}
x\in l^G_\infty(\Delta^m_G)&\Rightarrow \Delta^m_Gx \in l^G_\infty\\
                           &\Rightarrow \sup_k |\Delta^m_Gx _k|^G<\infty\\
													&\Rightarrow \sup_k |\Delta^{m-1}_Gx _k \ominus \Delta^{m-1}_Gx _{k+1}|^G<\infty\\
													&\Rightarrow \sup_k e^{k^{-1}}\odot |\Delta^m_Gx _k|^G<\infty \quad \text{by Corollary (\ref{corollary1})}\\
													&\Rightarrow \sup_k e^{k^{-m}}\odot |x _k|^G<\infty \quad \text{by Corollary (\ref{corollary2}).}
\end{align*}
\end{proof}\vspace{-0.4cm}
\section{$\alpha-,\beta-,  \gamma-$ duals}
\begin{defn} \cite{Garling67, KotheToplitz69, KotheToplitz34, Maddox80} If $X$ is a sequence space, it is defined that
\begin{enumerate}
\item[(i)] $X^\alpha=\{a=(a_k) : \sum_{k=1}^\infty |a_k x_k|<\infty, \, \text{for each} \,x\in X\};$
\item[(ii)] $X^\beta=\{a=(a_k) : \sum_{k=1}^\infty a_k x_k \, \text{is convergent, for each} \,x\in X\};$
\item[(iii)] $ X^\gamma=\{a=(a_k) : \sup_n|\sum_{k=1}^n a_k x_k|<\infty, \, \text{for each} \,x\in X\}.$
\end{enumerate}
\end{defn}

Then $X^\alpha, X^\beta$ and $X^\gamma$ are called $\alpha-$dual (or K\"{o}the-Toeplitz dual), $\beta-$dual (or generalized K\"{o}the-Toeplitz dual) and $\gamma-$dual spaces of $X,$ respectively. Then $X^\alpha \subset X^\beta \subset X^\gamma.$ If $X\subset Y,$ then $Y^\dag\subset X^\dag,$ for $\dag=\alpha, \beta$ or $\gamma.$ It is clear that $X\subset \left(X^\alpha \right)^\alpha=X^{\alpha \alpha}.$ If $X= X^{\alpha \alpha}$ then $X$ is called $\alpha-$space. $\alpha-$space is also called a K\"{o}the space or a perfect sequence space.
 
Then we defined and have proved that\cite{KhirodBipan}\\
\qquad $\left(sl_\infty^G(\Delta_G)\right)^\alpha=\left\{a=(a_k): {_G\sum_{k=1}^\infty} e^k\odot |a_k|^G<\infty\right\}$\\
\qquad $\left(sl_\infty^G(\Delta_G)\right)^\beta=  \left\{a=(a_k): {_G\sum_{k=1}^\infty} e^k\odot a_k \text{~is convergent with~} {_G\sum_{k=1}^\infty}|R_k|^G<\infty\right\}$\\
\qquad $\left(sl_\infty^G(\Delta_G)\right)^\gamma=\left\{a=(a_k): \sup_n|{_G\sum_{k=1}^n} e^k\odot a_k|^G<\infty, {_G\sum_{k=1}^\infty}|R_k|^G<\infty\right\}$\\
 where $R_k = {_G\sum_{n=k+1}^\infty} a_n$ and $s: l^G_\infty(\Delta_G)\rightarrow l^G_\infty(\Delta_G), x\rightarrow sx=y=(1, x_2, x_3,....).$
\begin{lemma}\label{lemma6th}
Let $U_1=\{a=(a_k):~ _G\sum_k e^{k^m}\odot |a_k|^G<\infty \}.$ Then $\left[Dl^G_\infty(\Delta^m_G)\right]^\alpha=U_1.$ 
\end{lemma}
\begin{proof}
Let $a\in U_1,$ then using Corollary (\ref{corollary1}) for $x\in Dl^G_\infty(\Delta^m_G),$ we have\\
$_G\sum_k|a_k\odot x_k|^G= ~ _G\sum_k \left\{e^{k^m}\odot |a_k|^G\right\} \odot \left\{e^{k^{-m}}\odot |x_k|^G\right\}<\infty $ by Corollary (\ref{corollary2}).\\
This implies that $a \in \left[Dl^G_\infty(\Delta^m_G)\right]^\alpha.$ Therefore
\begin{equation}\label{eqn2}
U_1 \subseteq \left[Dl^G_\infty(\Delta^m_G)\right]^\alpha.
\end{equation}

Conversely, let $a\in \left[Dl^G_\infty(\Delta^m_G)\right]^\alpha.$ Then $_G\sum_k|a_k\odot x_k|^G<\infty$ (by definition of $\alpha-$dual) for $x\in Dl^G_\infty(\Delta^m_G).$ So we take
\begin{equation}\label{eqn3}
x_k=
\begin{cases}
1, &\text{if} ~k\leq m\\
e^{k^m},&\text{if}~ k> m
\end{cases}
\end{equation}
Then $x=(1,1,1,...,1, e^{(m+1)^m}, e^{(m+2)^m},...)\in Dl^G_\infty(\Delta^m_G).$ Therefore
\begin{align*}
_G\sum^\infty_{k=1} e^{k^m}\odot|a_k|^G &=~ _G\sum^m_{k=1} e^{k^m}\odot|a_k|^G \oplus ~ _G\sum^\infty_{k=m+1} e^{k^m}\odot|a_k|^G\\
                                        &=~ _G\sum^m_{k=1} e^{k^m}\odot|a_k|^G \oplus ~ _G\sum^\infty_{k=1} |a_k\odot x_k|^G<\infty
\end{align*}
since $a_k\odot x_k=1$(the geometric zero) for $k=1, 2,...., m.$\\
Therefore $a\in U_1 .$ This implies
\begin{equation}\label{eqn4}
\left[Dl^G_\infty(\Delta^m_G)\right]^\alpha \subseteq U_1.
\end{equation}
Then from (\ref{eqn2}) and (\ref{eqn4}), we get
\[\left[Dl^G_\infty(\Delta^m_G)\right]^\alpha = U_1.\]
\end{proof}
\begin{lemma}\label{lemma7th}
\[\left[ Dl^G_\infty(\Delta^m_G)\right]^\alpha =\left[Dc^G(\Delta^m_G)\right]^\alpha.\]
\end{lemma}
\begin{proof}
Since $Dc^G(\Delta^m_G) \subseteq Dl^G_\infty(\Delta^m_G),$ hence $\left[ Dl^G_\infty(\Delta^m_G)\right]^\alpha \subseteq \left[Dc^G(\Delta^m_G)\right]^\alpha .$

Again let $a\in \left[Dc^G(\Delta^m_G)\right]^\alpha .$ Then $_G\sum_k|a_k\odot x_k|^G<\infty$ for each $x\in Dc^G(\Delta^m_G).$ If we take $x=(x_k)$ which is defined in (\ref{eqn3}), we get
\[_G\sum_k e^{k^m}\odot|a_k|^G=~ _G\sum^m_{k=1} e^{k^m}\odot|a_k|^G \oplus ~ _G\sum_k |a_k\odot x_k|^G<\infty.\]
This implies that $a\in \left[ Dl^G_\infty(\Delta^m_G)\right]^\alpha.$ Thus

 \[\left[ Dl^G_\infty(\Delta^m_G)\right]^\alpha= \left[ Dc^G(\Delta^m_G)\right]^\alpha.\]
\end{proof}
\begin{lemma}\label{lemma8th}
\begin{align*}
&(i)~ \left[l^G_\infty(\Delta^m_G)\right]^\alpha= \left[ Dl^G_\infty(\Delta^m_G)\right]^\alpha.\\
&(ii) \left[ c^G(\Delta^m_G)\right]^\alpha= \left[ Dc^G(\Delta^m_G)\right]^\alpha.
\end{align*}
\end{lemma}
\begin{proof}
$(i)$ Since $Dl^G_\infty(\Delta^m_G) \subseteq l^G_\infty(\Delta^m_G),$ so $\left[l^G_\infty(\Delta^m_G)\right]^\alpha \subseteq\left[ Dl^G_\infty(\Delta^m_G)\right]^\alpha.$

Let $a \in \left[ Dl^G_\infty(\Delta^m_G)\right]^\alpha$ and $x\in l^G_\infty(\Delta^m_G).$ From Corollary (\ref{corollary3}), we have
\[_G\sum_k|a_k\odot x_k|^G= ~ _G\sum_k e^{k^m}\odot |a_k|^G \odot (e^{k^{-m}}\odot |x_k|^G)<\infty .\]
Hence $a \in \left[ l^G_\infty(\Delta^m_G)\right]^\alpha.$

$(ii) ~ Dc^G(\Delta^m_G) \subseteq c^G(\Delta^m_G)$ implies $\left[ c^G(\Delta^m_G)\right]^\alpha \subseteq \left[ Dc^G(\Delta^m_G)\right]^\alpha.$

Let $a\in \left[ Dc^G(\Delta^m_G)\right]^\alpha$ and $x\in c^G(\Delta^m_G).$ From Corollary (\ref{corollary3}), we have
\[_G\sum_k|a_k\odot x_k|^G= ~ _G\sum_k e^{k^m}\odot |a_k|^G \odot (e^{k^{-m}}\odot |x_k|^G)<\infty \]
for $x\in c^G(\Delta^m_G)\subseteq l^G_\infty(\Delta^m_G).$ Hence $a\in \left[c^G(\Delta^m_G)\right]^\alpha.$ This completes the proof.
\end{proof}
\begin{thm}\label{theorem2nd}
Let $X$ stand for $l^G_\infty$ or $c^G.$ Then
\[\left[X(\Delta^m_G)\right]^\alpha = \{a=(a_k):~ _G\sum_k e^{k^m}\odot |a_k|^G<\infty \}.\]
\end{thm}
\begin{proof}
\begin{align*}
\left[l^G_\infty(\Delta^m_G)\right]^\alpha &= \left[ Dl^G_\infty(\Delta^m_G)\right]^\alpha \qquad\text{by Lemma (\ref{lemma8th})}\\
                                           &= \{a=(a_k):~ _G\sum_k e^{k^m}\odot |a_k|^G<\infty \} \text{~by Lemma (\ref{lemma6th})}.
\end{align*}
Again
\begin{align*}
\left[ c^G(\Delta^m_G)\right]^\alpha &= \left[ Dc^G(\Delta^m_G)\right]^\alpha \qquad\text{by Lemma (\ref{lemma8th})}\\
																		&= \left[ Dl^G_\infty(\Delta^m_G)\right]^\alpha \qquad\text{by Lemma (\ref{lemma7th})}\\
                                     &= \{a=(a_k):~ _G\sum_k e^{k^m}\odot |a_k|^G<\infty \} \text{~by Lemma (\ref{lemma6th})}.
\end{align*}
\end{proof}
\begin{corollary}
For $X = l^G_\infty$ or $c^G,$ we have
\[\left[X(\Delta_G)\right]^\alpha = \{a=(a_k):~ _G\sum_k e^k\odot |a_k|^G<\infty \}, \text{~and}\]
\[\left[X(\Delta^2_G)\right]^\alpha = \{a=(a_k):~ _G\sum_k e^{k^2}\odot |a_k|^G<\infty \}.\]
\end{corollary}
\begin{proof}
Putting $m=1$ and $m=2$ in Theorem (\ref{theorem2nd}),the results follow.
\end{proof}
\begin{thm}\label{theorem3rd}
Let $X$ stand for $l^G_\infty$ or $c^G$ and $U_2= \{a=(a_k): \sup_k e^{k^{-m}}\odot |a_k|^G<\infty \}.$ Then $\left[X(\Delta^m_G)\right]^{\alpha \alpha} = U_2.$
\end{thm}
\begin{proof}
Let $a\in U_2$ and $x\in \left[X(\Delta^m_G)\right]^\alpha,$ then by definition of $U_2$ and by Lemma (\ref{lemma6th}), we get
\begin{align*}
 _G\sum_k|a_k\odot x_k|^G &=_G\sum_k e^{k^m}\odot|x_k|^G \odot e^{k^{-m}}\odot|a_k|^G\\
                          & \leq _G\sum_k e^{k^m}\odot|x_k|^G\odot\sup_k e^{k^{-m}}\odot|a_k|^G<\infty.
\end{align*}
Hence $a\in \left[X(\Delta^m_G)\right]^{\alpha \alpha}.$

Conversely, let $a\in \left[X(\Delta^m_G)\right]^{\alpha \alpha}$ and $a \notin U_2.$ Then we must have
\[\sup_k e^{k^{-m}}\odot|a_k|^G =\infty. \]

Hence there exists a strictly increasing sequence $(e^{k(i)})$ of geometric integers\cite{TurkmenBasar}, where $k(i)$ is a strictly increasing sequence of positive integers such that
\[e^{[{k(i)}]^{-m}}\odot|a_{k(i)}|^G> e^{i^m}.\]
Let us define the sequence $x$ by
\begin{equation*}
x_k={
\begin{cases}
\left(|a_{k(i)}|^G\right)^{-1_G},&~k=k(i)\\
1, &k \neq k(i).
\end{cases}}
\end{equation*}
where $\left(|a_{k(i)}|^G\right)^{-1_G}$ is the geometric inverse of $|a_{k(i)}|^G$ so that $|a_{k(i)}|^G \odot \left(|a_{k(i)}|^G\right)^{-1_G}=e.$\\
Then we have
\[_G\sum_ke^{k^m}\odot |x_k|^G=~_G\sum_i e^{[{k(i)}]^m}\odot\left[|a_{k(i)}|^G\right]^{-1_G}\leq e^{i^{-m}}<\infty.\]
Hence $x\in \left[X(\Delta^m_G)\right]^\alpha $ and $_G\sum_k |a_k\odot x_k|^G= \sum e=\infty.$
This is a contradiction as  $a\in \left[X(\Delta^m_G)\right]^{\alpha \alpha}.$ Hence $a\in U_2.$
\end{proof}
\begin{corollary}
For $X = l^G_\infty$ or $c^G,$ we have
\[\left[X(\Delta^2_G)\right]^{\alpha \alpha} = \{a=(a_k): \sup_k e^{k^{-2}}\odot |a_k|^G<\infty \}.\]
\end{corollary}
\begin{proof}
In Theorem (\ref{theorem3rd}), putting $m=2$ we obtain the result.
\end{proof}
\begin{corollary}
The sequence spaces $l^G_\infty(\Delta^m_G)$ and $c^G(\Delta^m_G)$ are not perfect.
\end{corollary}
\begin{proof}
Proof is trivial as $X^{\alpha \alpha} \neq X$ for $X= l^G_\infty(\Delta^m_G)$ or $c^G(\Delta^m_G).$
\end{proof}

\end{document}